\documentclass[noinfoline]{article}

\RequirePackage[OT1]{fontenc}
\RequirePackage[
amsthm,amsmath
]{imsart}

\usepackage{graphicx}
\usepackage{latexsym,amsmath}
\usepackage{amsmath,amsthm,amscd}
\usepackage{amsfonts}
\usepackage[psamsfonts]{amssymb}
\usepackage{enumerate}

\usepackage{url}
\usepackage{tocvsec2}

\usepackage{epstopdf}
\usepackage{float}

\startlocaldefs

\theoremstyle{plain} 
\newtheorem{theorem}{Theorem}
\newtheorem{corollary}[theorem]{Corollary}

\newtheorem{lemma}[theorem]{Lemma}

\newtheorem{proposition}[theorem]{Proposition}

\theoremstyle{definition} 

\theoremstyle{definition} 
\newtheorem{ex}[theorem]{Example}
\newtheorem*{ex*}{Example}
\theoremstyle{remark} 

\theoremstyle{remark} 

\newtheorem*{remark*}{Remark}

\newcommand{\dd}{\partial}

\renewcommand{\dd}{{\operatorname{d}}}

\newcommand{\supp}{\operatorname{supp}}
\newcommand{\card}{\operatorname{card}}
\renewcommand{\ex}{\operatorname{ex}}

\newcommand{\al}{\alpha}

\newcommand{\si}{\sigma}
\newcommand{\Si}{\Sigma}

\newcommand{\la}{\lambda}
\renewcommand{\la}{a}

\newcommand{\de}{\delta}
\newcommand{\be}{\beta}
\newcommand{\De}{\Delta}
\newcommand{\La}{\Lambda}

\newcommand{\ii}[1]{\,\mathbf{I}\{#1\}}

\newcommand{\intr}[2]{\overline{#1,#2}}

\newcommand{\f}{\mathbf{f}}

\newcommand{\Z}{\mathbb{Z}}
\newcommand{\R}{\mathbb{R}}
\newcommand{\B}{\mathcal{B}}
\newcommand{\C}{\mathfrak{C}}
\renewcommand{\C}{\Gamma}
\newcommand{\F}{\mathcal{F}}

\newcommand{\vp}{\varepsilon}

\newcommand{\tmu}{{\tilde{\mu}}}
\newcommand{\tnu}{{\tilde{\nu}}}

\newcommand{\msi}{{\mathrm{M}}}
\newcommand{\K}{{\mathrm{K}}}

\newcommand{\Dekf}{{\De_\f^{(k)}}}
\newcommand{\Dek}{{\De^{(k)}}}

\newcommand{\N}{{\mathbb{N}}}

\renewcommand{\le}{\leqslant}
\renewcommand{\ge}{\geqslant}

\endlocaldefs

\begin{document}

\noindent \jobname.tex;  
\today

\begin{frontmatter}

\title{On the extreme points of moments sets}
\runtitle{Extreme points of moments sets}

%

\begin{aug}
\author{\fnms{Iosif} \snm{Pinelis}\thanksref{t2}\ead[label=e1]{ipinelis@mtu.edu}
}
  \thankstext{t2}{Supported by NSA grant H98230-12-1-0237}
\runauthor{Iosif Pinelis}


\address{Department of Mathematical Sciences\\
Michigan Technological University\\
Houghton, Michigan 49931, USA\\
E-mail: \printead[ipinelis@mtu.edu]{e1}}
\end{aug}

\begin{abstract} 
Necessary and sufficient conditions for a measure to be an extreme point of the set of measures (on an abstract measurable space) with prescribed generalized moments 
are given, as well as an application to extremal problems over such moment sets;  
{these conditions are expressed in terms of atomic partitions of the measurable space. It is also shown that every such extreme measure can be adequately represented by 
a linear combination of $k$ Dirac probability measures with nonnegative coefficients, where $k$ is the number of restrictions on moments; moreover, when the measurable space has appropriate topological properties, the phrase ``can be adequately represented by'' here can be replaced simply by ``is''.}
The proofs are elementary and mainly self-contained. 
\end{abstract}

  
%



\end{frontmatter}

\settocdepth{chapter}


\settocdepth{subsubsection}

\theoremstyle{plain} 


Let $S$ be an arbitrary set and let $\Si$ be a $\si$-algebra over $S$, so that $(S,\Si)$ is a measurable space. 
Let $\msi$ denote the set of all (nonnegative) real-valued measures on $\Si$. 
For any $\mu\in\msi$ and any $A\in\Si$, define the truncation $\mu_A$ of the measure $\mu$ by $A$ via the formula 
\begin{equation*}
	\mu_A(B):=\mu(A\cap B) 
\end{equation*}
for all $B\in\Si$; it is then clear that $\mu_A\in\msi$. 

For any $\mu\in\msi$ and any $A\in\Si$, let us say that $A$ is a \emph{$\mu$-atom} if $\mu_A(B)\in\{0,\mu(A)\}$ for all $B\in\Si$. 

For any $n\in\N$ and $\mu\in\msi$, let us refer to an $n$-tuple $(A_1,\dots,A_n)$ of members of $\Si$ as a \emph{non-null $(n,\mu)$-partition of $S$} if the sets $A_1,\dots,A_n$ are pairwise disjoint, $A_1\cup\dots\cup A_n=S$, and $\mu(A_i)>0$ for all $i\in\intr1n$; 
let us say that such a partition $(A_1,\dots,A_n)$ is \emph{atomic} if $A_i$ is a $\mu$-atom for each $i\in\intr1n$; 
here and in what follows, for any $m$ and $n$ in $\Z\cup\{\infty\}$ we let $\intr mn:=\{j\in\Z\colon m\le j\le n\}$. 

Let $\F$ stand for the set of all $\Si$-measurable real-valued functions on $S$. 

Take any $k\in\intr1\infty$. For each $j\in\intr1k$, take any $f_j\in\F$, and let 
\begin{equation*}
	\f:=(f_1,\dots,f_k). 
\end{equation*}
For any $\mu\in\msi$, let us write $\f\in L^1(\mu)$ if $f_j\in L^1(S,\Si,\mu)$ for all $j\in\intr1k$, and also let 
\begin{equation*}
\int_A \f\dd\mu:=\Big(\int_A f_1\dd\mu,\dots,\int_A f_k\dd\mu\Big) 	
\end{equation*}
for all $\f\in\ L^1(\mu)$ and $A\in\Si$. 

For any $C\subseteq\R^k$, consider the moment set  
\begin{equation*}
	\msi_{\f,C}:=\Big\{\mu\in\msi\colon \f\in L^1(\mu), \int_S \f\dd\mu\in C\Big\}.  
\end{equation*}
In the case when $C$ is a singleton set of the form $\{c\}$ for some $c\in\R^k$, let us write $\msi_{\f,c}$ 
in place of $\msi_{\f,C}$. 

Let $\La$ denote a subset of $\msi$, and then consider the corresponding narrower moment sets  
\begin{equation}\label{eq:La}
\La_{\f,C}:=\La\cap\msi_{\f,C}\quad\text{and}\quad\La_{\f,c}:=\La\cap\msi_{\f,c},  
\end{equation}
again for any $C\subseteq\R^k$ and $c\in\R^k$. 

Consider also 
\begin{equation*}
	\Pi:=\{\mu\in\La\colon\mu(S)=1\}, 
\end{equation*}
the set of all probability measures in $\La$. 

For any (not necessarily convex) subset $\K$ of $\msi$, let $\ex\K$ denote the set of all extreme points of $\K$. Thus, $\mu\in\ex\K$ if and only if $\mu\in\K$ and for any $(t,\nu_0,\nu_1)\in(0,1)\times\K\times\K$ such that 
$(1-t)\nu_0+t\nu_1=\mu$ one has $\nu_0=\nu_1$. 

The following theorem presents a necessary condition for a measure to be an extreme point of the set $\La_{\f,C}$. 
 
\begin{theorem}\label{th:necess} \emph{\bf (Necessity).} 
Take any $C\subseteq\R^k$. 
Suppose that the following condition holds: 
\begin{equation}\label{eq:cond}
	\text{$\La$ is a convex cone such that $\mu_A\in\La$ for all $\mu\in\La$ and $A\in\Si$.}  
\end{equation}
Take any $\mu\in\ex\La_{\f,C}\setminus\{0\}$. Then 
for some $m\in\intr1k$ 
there is an atomic non-null $(m,\mu)$-partition $(A_1,\dots,A_m)$ of $S$ such that the vectors $\int_{A_1} \f\dd\mu,\dots,\int_{A_m} \f\dd\mu$ are linearly independent. 
\end{theorem}  


As usual, let us say that a measure $\mu\in\msi$ is a $0,\!1$ measure if all its values are in the set $\{0,1\}$; so, any $0,\!1$ measure in $\La$ is also in $\Pi$. 

In the case when $k=1$, $f_1=1$ on 
$S$, and $C=\{1\}\subset\R^1$, Theorem~\ref{th:necess} turns into 

\begin{corollary}\label{cor:necess,k=1} 
Suppose that condition \eqref{eq:cond} holds. 
Then any measure in 
$\ex\Pi$ is a $0,\!1$ measure. 
\end{corollary} 

\begin{proof}[Proof of Theorem~\ref{th:necess}]
Suppose that for some $n\in\intr{k+1}\infty$ there is a non-null $(n,\mu)$-partition $(A_1,\dots,A_n)$ of $S$, so that $\mu(A_i)>0$ for all $i\in\intr1n$. Since $n>k$, the $n$ vectors $\int_{A_1} \f\dd\mu,\dots,\int_{A_n} \f\dd\mu$ in $\R^k$ are linearly dependent, so that $\vp_1\int_{A_1} \f\dd\mu+\dots+\vp_n\int_{A_n} \f\dd\mu=0$ for some nonzero vector $\vp=(\vp_1,\dots,\vp_n)\in(-1,1)^n$. For any such vector $\vp$, let 
$\nu_\pm:=(1\pm\vp_1)\mu_{A_1}+\dots+(1\pm\vp_n)\mu_{A_n}$. 
Then, by \eqref{eq:cond}, $\nu_\pm\in\La$. Moreover, $\int_S \f\dd\nu_\pm=
\sum_1^n(1\pm\vp_i)\int_{A_i}\f\dd\mu=\sum_1^n\int_{A_i}\f\dd\mu
=\int_S \f\dd\mu\in C$, so that $\nu_\pm\in\La_{\f,C}$. 
Also, $\frac12(\nu_++\nu_-)=\mu$ and $(\nu_+-\nu_-)(A_i)=2\vp_i\mu(A_i)\ne0$ for some $i\in\intr1n$. 
So, $\mu\notin\ex\La_{\f,C}$, which is a contradiction. 

Therefore, $N_\mu\subseteq\intr1k$, where $N_\mu$ denotes the set of all $n\in\N$ for which there is a non-null $(n,\mu)$-partition of $S$. On the other hand, the condition $\mu\ne0$ implies that 
$1\in N_\mu$, so that $N_\mu\ne\emptyset$. 
Thus, $m:=\max N_\mu\in\intr1k$. Since $m\in N_\mu$, there is a non-null $(m,\mu)$-partition $(A_1,\dots,A_m)$ of $S$, and the reasoning in the previous paragraph shows that the vectors $\int_{A_1} \f\dd\mu,\dots,\int_{A_m} \f\dd\mu$ are necessarily linearly independent.  
It remains to show that the non-null $(m,\mu)$-partition $(A_1,\dots,A_m)$ of $S$ is atomic. 
Indeed, suppose the contrary, so that, without loss of generality, $A_1$ is not a $\mu$-atom. Then for some $B\in\Si$ one has $B\subset A_1$ and $0<\mu(B)<\mu(A_1)$, and so, $(B,A_1\setminus B,A_2,\dots,A_m)$ is a non-null $(m+1,\mu)$-partition of $S$, which contradicts the definition $m:=\max N_\mu$. 
Now the proof of Theorem~\ref{th:necess} is complete. 
\end{proof}


Assuming the conditions of Theorem~\ref{th:necess} and using it, we shall show that every measure $\mu\in\ex\La_{\f,C}$ can be represented, in a certain sense sufficient for applications, by a linear combination of Dirac probability measures with nonnegative coefficients. 
Recall that the Dirac probability measure 
$\de_s$ with the mass at a point $s\in S$ is defined by the formula $\de_s(A)=\ii{s\in A}$ for all $A\in\Si$, where $\ii{\cdot}$ is the indicator function.  
Introduce indeed the following sets of (discrete) measures on $\Si$: 
\begin{align*}
	\Dekf&:=\Big\{\sum_1^m \la_i\de_{s_i}\colon 
	m\in\intr0k, 
	(\la_1,\dots,\la_m)\in(0,\infty)^m, 
	(s_1,\dots,s_m)\in S^m,  \\ 
&\qquad\qquad\qquad\qquad\qquad\quad	\f(s_1),\dots,\f(s_m)\text{ are linearly independent}\Big\},  \\ 
\intertext{and}
	\Dek&:=\Big\{\sum_1^k \la_i\de_{s_i}\colon 
	(\la_1,\dots,\la_k)\in[0,\infty)^k, 
	(s_1,\dots,s_k)\in S^k\Big\}. 
\end{align*}
In the above definition of $\Dekf$, in the case when $m=0$ it is assumed that $\sum_1^m \la_i\de_{s_i}$ is (the) zero (measure) and that the condition following $m\in\intr0k$ between the braces is trivially satisfied. In particular, 
$
	0\in\Dekf\subseteq\Dek\subseteq\msi.   
$ 


Note that for any $g\in\F$ and 
\begin{equation*}
	\text{for any $\nu=\sum_1^k \la_i\de_{s_i}\in\Dek$, one has 
	$\int_S g\dd\nu=\sum_1^k \la_i g(s_i)$.} 
\end{equation*}
To obtain Corollary~\ref{cor:dirac repr} below, we shall need
\begin{lemma}\label{lem:const} 
Take any $\mu\in\msi$. 
If $A$ is a $\mu$-atom 
and $g\in\F$, 
then there is a real number $\al$ such that $g=\al$ $\mu$-almost everywhere (a.e.) on $A$. 
\end{lemma}

\begin{proof}[Proof of Lemma~\ref{lem:const}]
For all $\be\in[-\infty,\infty]$, let $A_\be:=\{s\in A\colon g(s)<\be\}$. 
Then $\mu(A_\be)$ is nondecreasing in $\be\in\R$ and takes values in the set $\{0,\mu(A)\}$. 
Let $\be_*:=\sup\{\be\in\R\colon\mu(A_\be)=0\}\in[-\infty,\infty]$. 
Then $\mu(A_\be)=0$ for all $\be\in[-\infty,\be_*)$ and $\mu(A_\be)=\mu(A)$ for all $\be\in(\be_*,\infty]$. 
If $\be_*>-\infty$ then $\mu(A_{\be_*})=\lim_{\be\uparrow\be_*}\mu(A_\be)=0$; if $\be_*=-\infty$ then $A_{\be_*}=\emptyset$ and hence again $\mu(A_{\be_*})=0$. 
Thus, in all cases $\mu(A_{\be_*})=0$ or, equivalently, $g\ge\be_*$ $\mu$-a.e.\ on $A$.  
Similarly, if $\be_*<\infty$ then $\mu(A_{\be_*+})=\lim_{\be\downarrow\be_*}\mu(A_\be)=\mu(A)$, where $A_{\be+}:=\{s\in A\colon g(s)\le\be\}$; if $\be_*=\infty$ then $A_{\be_*+}=A$ and hence again $\mu(A_{\be_*})=\mu(A)$. 
Thus, in all cases $g\le\be_*$ $\mu$-a.e.\ on $A$.  
We conclude that $g=\be_*$ $\mu$-a.e.\ on $A$. 
In particular, it follows that, if $\mu(A)>0$, then necessarily $\be_*\in\R$, because $g$ is real-valued; and if $\mu(A)=0$ then $g=\al$ $\mu$-a.e.\ on $A$ for any given real number $\al$. 
\end{proof}

From now on, take any $C\subseteq\R^k$ and take $g$ to be any function in $\F$ such that the integral $\int_S g\dd\mu$ exists in $[-\infty,\infty]$ for each $\mu\in\La_{\f,C}$. 

\begin{corollary}\label{cor:dirac repr} 
Suppose that condition \eqref{eq:cond} holds. 
Take any 
$\mu\in\ex\La_{\f,C}$. 
Then 
\begin{enumerate}[(i)]
	\item there is a measure $\tmu$ such that 
\begin{equation}\label{eq:f repr}
	\tmu\in\Dekf\cap\msi_{\f,C},\quad \int_S \f\dd\tmu=\int_S \f\dd\mu,\quad\text{and}\quad 
	\int_S g\dd\tmu=\int_S g\dd\mu.  
\end{equation}
	\item
If, in addition, 
\begin{equation}\label{eq:dirac in}
	\Dekf\subseteq\La, 
\end{equation}
then it follows that $\tmu\in\Dekf\cap\La_{\f,C}$. 
\end{enumerate} 

In this corollary, one can replace $\Dekf$ by $\Dek$ throughout. 
\end{corollary} 

\begin{proof}[Proof of Corollary~\ref{cor:dirac repr}]
If $\mu=0$ then, as discussed above, $\mu\in\Dekf$, 
so that \eqref{eq:f repr} will hold with $\tmu=\mu$. 
It remains to consider the case $\mu\in\ex\La_{\f,C}\setminus\{0\}$. 
Then, by Theorem~\ref{th:necess}, for some $m\in\intr1k$ 
there is an atomic non-null $(m,\mu)$-partition $(A_1,\dots,A_m)$ of $S$ such that the vectors $\int_{A_1} \f\dd\mu,\dots,\int_{A_m} \f\dd\mu$ are linearly independent. 

Let now $f_0:=g$. 
By Lemma~\ref{lem:const}, for each pair $(i,j)\in\intr1m\times\intr0k$ there exist a subset $B_{i,j}$ of $A_i$ and a real number $\la_{i,j}$ such that $B_{i,j}\in\Si$, $\mu(B_{i,j})=\mu(A_i)$, and $f_j(s)=\la_{i,j}$ for all $s\in B_{i,j}$. 

Fix, in this paragraph, any $i\in\intr1m$, and let $B_i:=\cap_{j=0}^k B_{i,j}$. Then $B_i\subseteq A_i$, $\mu(B_i)=\mu(A_i)>0$ and $f_j(s)=\la_{i,j}$ for all $s\in B_i$ and $j\in\intr0k$. Since $\mu(B_i)>0$, there is some $s_i\in B_i$. For any such $s_i$ and each $j\in\intr0k$, 
\begin{equation*}
\int_{A_i}f_j\dd\mu=\la_{i,j}\mu(A_i)=f_j(s_i)\mu(A_i)=\int_{A_i}f_j\dd\tmu,    	
\end{equation*}
where $\tmu:=\mu(A_1)\de_{s_1}+\dots+\mu(A_m)\de_{s_m}$.
Now \eqref{eq:f repr} immediately follows, and then  
part (ii) of Corollary~\ref{cor:dirac repr} immediately follows by \eqref{eq:La}. 
%

The last sentence of the corollary now follows immediately as well, because $\Dekf\subseteq\Dek$. 
\end{proof}


\begin{corollary}\label{cor:dirac opt} 
Suppose that condition \eqref{eq:cond} holds and 
\begin{equation}\label{eq:=ex}
	\sup\Big\{\int_S g\dd\mu\colon{\mu\in\La_{\f,C}}\Big\} =
	\sup\Big\{\int_S g\dd\mu\colon{\mu\in\ex\La_{\f,C}}\Big\}.
\end{equation}
Then 
\begin{equation}\label{eq:dirac opt le}
	\sup\Big\{\int_S g\dd\mu\colon{\mu\in\La_{\f,C}}\Big\} \le 
	\sup\Big\{\int_S g\dd\mu\colon{\mu\in\Dekf\cap\msi_{\f,C}}\Big\}.  
\end{equation}  
If condition \eqref{eq:dirac in} holds as well, then  
\begin{equation}\label{eq:dirac opt eq}
	\sup\Big\{\int_S g\dd\mu\colon{\mu\in\La_{\f,C}}\Big\} = 
	\sup\Big\{\int_S g\dd\mu\colon{\mu\in\Dekf\cap\La_{\f,C}}\Big\}.  
\end{equation}  
%

In \eqref{eq:dirac opt le}, one can replace $\Dekf$ by $\Dek$; one can do so in \eqref{eq:dirac opt eq} too if such a replacement is done in condition \eqref{eq:dirac in} as well. 
\end{corollary} 

\begin{proof}[Proof of Corollary~\ref{cor:dirac opt}] 
By part (i) of Corollary~\ref{cor:dirac repr}, 
the right-hand side of \eqref{eq:=ex} is no greater than the right-hand side of \eqref{eq:dirac opt le}. 
Hence, by \eqref{eq:=ex}, the left-hand side of \eqref{eq:dirac opt le} is no greater than its right-hand side. 

Now, if \eqref{eq:dirac in} holds, then 
$\Dekf\cap\msi_{\f,C}=\Dekf\cap\La_{\f,C}$, and hence the right-hand side of \eqref{eq:dirac opt le} 
equals the right-hand side of \eqref{eq:dirac opt eq}, which in turn is obviously no greater than the left-hand side of \eqref{eq:dirac opt eq}. 
So, \eqref{eq:dirac opt eq} follows. 

The last sentence of Corollary~\ref{cor:dirac opt} is proved quite similarly, using the last sentence of Corollary~\ref{cor:dirac repr}. 
\end{proof} 

Applications of equalities of the form \eqref{eq:dirac opt eq} can be found e.g.\ in \cite{pin-utev-exp}. 

Let us now indicate a number of generic cases when 
condition \eqref{eq:=ex} holds: 

\begin{proposition}\label{prop:=ex}
Condition \eqref{eq:=ex} is satisfied in each of the following cases:  
\begin{enumerate}[(i)]
	\item when there exists an extreme point of the set 
\begin{equation*}
	\La_{\max g;\;\f,C}:=\Big\{\nu\in\La_{\f,C}\colon\int_S g\dd\nu\ge\int_S g\dd\mu\text{ for all }\mu\in\La_{\f,C}\Big\}; 
\end{equation*}
	\item when $\La_{\max g;\;\f,C}$ is a nonempty compact convex subset of a locally convex space; 
	\item when $\La_{\max g;\;\f,C}$ is a nonempty compact finite-dimensional set; 
	\item when $\La_{\max g;\;\f,C}$ is a singleton set \big(that is, when the maximum of $\int_S g\dd\mu$ over all $\mu\in\La_{\f,C}$ is attained at a unique measure $\mu\in\La_{\f,C}$\big); 
	\item when the set $S$ is endowed with the structure of a Hausdorff topological space, the $\si$-algebra $\Si$ coincides with the corresponding Borel $\si$-algebra $\B$, $\La=\msi$, and $C=\{1\}\times I_2\times\dots\times I_k$, where $I_2,\dots,I_k$ are arbitrary closed convex subsets of $\R$ \big(so that all measures in $\La_{\f,C}$ are probability measures\big).  
\end{enumerate}
\end{proposition}

\begin{proof}[Proof of Proposition~\ref{prop:=ex}] \ 

\noindent (i):\quad Suppose that condition (i) of Proposition~\ref{prop:=ex} holds, so that there exists an extreme point of the set $\La_{\max g;\;\f,C}$. Then it is easy to see that any such point (say $\mu_{\max}$) is in $\ex\La_{\f,C}$. At that, $\int_S g\dd\mu_{\max}$ equals the left-hand side of \eqref{eq:=ex}. Thus, \eqref{eq:=ex} follows. 

\noindent (ii,\;iii,\;iv):\quad Suppose that condition (ii) of Proposition~\ref{prop:=ex} holds. 
Then, by the Krein--Milman theorem (see e.g.\ \cite{phelps}), condition (i) of the proposition holds as well. Thus, \eqref{eq:=ex} follows. 
Note also that condition (iv) of Proposition~\ref{prop:=ex} implies condition (iii), which in turn implies (ii). 

\noindent (v):\quad Suppose that condition (v) of Proposition~\ref{prop:=ex} holds. Then, by the arguments in  \cite[Theorems~3.1 and 3.2 and Proposition~3.1]{winkler88} (which in turn rely mainly on \cite{weiz-wink}), \eqref{eq:=ex} follows. 
\end{proof}

Let us now supplement the results presented above by a few other ones, which are perhaps of lesser interest, in that they are not needed for applications such as Corollary~\ref{cor:dirac opt}. 

The following theorem supplements Theorem~\ref{th:necess}, as it presents a sufficient condition for a measure to be extreme in the moment set. Note that condition~\eqref{eq:cond} is not needed in Theorem~\ref{th:suff}; on the other hand, here $C$ is taken to be a singleton set. 
 
\begin{theorem}\label{th:suff} \emph{\bf (Sufficiency).} 
Take any $c\in\R^k$. 
Take any $\mu\in\La_{\f,c}
$ such that for some $m\in\intr1k$ 
there is an atomic non-null $(m,\mu)$-partition $(A_1,\dots,A_m)$ of $S$ and at that the vectors $\int_{A_1} \f\dd\mu,\dots,\int_{A_m} \f\dd\mu$ are linearly independent. 
Then $\mu\in\ex\La_{\f,c}$. 
\end{theorem} 

In the case when $k=1$, $f_1=1$ on 
$S$, and $c=1$, Theorem~\ref{th:suff} turns into 

\begin{corollary}\label{cor:suff,k=1} 
Any $0,\!1$ measure in $\Pi$ is in 
$\ex\Pi$ \emph{(cf.\ Corollary~\ref{cor:necess,k=1})}. 
\end{corollary} 

\begin{proof}[Proof of Theorem~\ref{th:suff}]
Take any 
$(t,\nu_+,\nu_-)\in(0,1)\times\La_{\f,c}\times\La_{\f,c}$ such that \break 
$(1-t)\nu_++t\nu_-=\mu$. 
We need to show that then $\nu_+=\nu_-$. 

\emph{Step 1.} Let us note here that $\nu_\pm(B)=0$ for all $B\in\Si$ such that $\mu(B)=0$; that is, the measures $\nu_\pm$ are absolutely continuous with respect to $\mu$. This follows because $1-t>0$, $t>0$, and for all $B\in\Si$ one has $0\le(1-t)\nu_+(B)\le\mu(B)$ and $0\le t\nu_-(B)\le\mu(B)$. 

\emph{Step 2.} Here we note that, if $A$ is a $\mu$-atom and a measure $\nu\in\msi$ is absolutely continuous with respect to $\mu$, then $A$ is a $\nu$-atom as well and, moreover, $\nu_A=\la\mu_A$ for some $\la\in[0,\infty)$.  
Indeed, take any $B\in\Si$ such that $B\subseteq A$. 
If $\mu(B)=0$ then $\nu(B)=0$, by the absolute continuity. Otherwise, $\mu(B)>0$ and $\mu(A\setminus B)=0$,  whence $\nu(A\setminus B)=0$; so, $\nu(B)=\nu(A)$ and $\mu(B)=\mu(A)>0$. 
Thus, $A$ is a $\nu$-atom and $\nu_A=\la\mu_A$, where $\la:=0$ if $\mu(A)=0$ and $\la:=\nu(A)/\mu(A)$ otherwise.  

\emph{Step 3.} From Steps 1 and 2 it follows that $(\nu_\pm)_{A_i}=\la_{\pm;i}\,\mu_{A_i}$ for every $i\in\intr1m$ and some $\la_{\pm;i}\in[0,\infty)$. 
Therefore and because $\nu_\pm\in\La_{\f,c}$, one has 
$0=c-c=\int_S \f\dd(\nu_+-\nu_-)
=\sum_1^m(\la_{+;i}-\la_{-;i})\int_{A_i}\f\dd\mu$. 
Therefore and because the vectors $\int_{A_1} \f\dd\mu,\dots,\int_{A_m} \f\dd\mu$ are linearly independent, $\la_{+;i}-\la_{-;i}=0$ and hence $(\nu_+)_{A_i}=(\nu_-)_{A_i}$ for all $i\in\intr1m$, which implies that indeed $\nu_+=\nu_-$.   
Now the proof of Theorem~\ref{th:suff} is complete. 
\end{proof}  

When the measurable space $(S,\Si)$ has to do with a topology, usually one can somewhat simplify the condition in Theorems~\ref{th:necess} and \ref{th:suff} of the existence of an atomic partition. 
More specifically, recall that by Corollary~\ref{cor:dirac repr} under condition \eqref{eq:cond} any measure $\mu\in\ex\La_{\f,C}$ can be represented  by a discrete measure $\tmu\in\Dekf$ in the sense of \eqref{eq:f repr}. That was enough for applications presented in Corollary~\ref{cor:dirac opt}. 
Yet, it may be of interest to know under what conditions any measure in the set $\ex\La_{\f,C}$ is (not just represented by but) equal to a discrete measure in $\Dekf$. 

To address this matter, suppose from now on to the rest of the paper that the set $S$ is endowed with a Hausdorff topology, and let then the $\si$-algebra $\Si$ contain 
the corresponding Borel $\si$-algebra $\B$. 


\begin{proposition}\label{prop:dirac} 
Take any $\mu\in\msi$ and $A\in\Si$. 
Then the following statements hold. 
\begin{enumerate}[(i)] 
\item
If $A$ is a $\mu$-atom then $\card\supp\mu_A\le1$. 
\item If $\card\supp\mu_A\le1$, $A\ne\emptyset$, and the 
measure $\mu_A$ is support-concentrated \big(that is, $\mu_A(\supp\mu_A)=\mu(A)$\big), 
then $\mu_A=\la\de_s$ for some $s\in A$ and $\la\in[0,\infty)$. 
\item For the 
measure $\mu_A$ to be support-concentrated, 
it is enough that the measure $\mu$ be a Radon one: $\mu(E)=\sup\{\mu(K)\colon K\subseteq E, K\text{ is compact}\}$ for all $E\in\Si$. 
\item 
In turn, for every measure in $\msi$ to be a Radon one, it is enough that $\Si=\B$ and 
$S$ be a Polish space (that is, a separable completely metrizable topological space); for instance, $\R^d$ or, more generally, any separable Banach space is Polish. 
\item If $\mu_A=\la\de_s$ for some $s\in A$ and $\la\in[0,\infty)$, then $A$ is a $\mu$-atom.   
\end{enumerate}
\end{proposition} 
Here, as usual, $\card$ stands for the cardinality and $\supp$ for the support. Recall that the support, $\supp\mu$, of a measure $\mu\in\msi$ is the set of all points $s\in S$ such that $\mu(G)>0$ for all open subsets $G$ of $S$ such that $G\ni s$; equivalently, $\supp\mu$ is the intersection of all closed subsets of $S$ of ``full'' measure $\mu(S)$.  

An immediate corollary of parts (i)--(iii) of Proposition~\ref{prop:dirac} is 
the well-known fact that every $0,\!1$-valued regular Borel measure on a Hausdorff space is a Dirac measure; 
see e.g.\ \cite[Corollary~2.4]{adamski76}. 
As seen from \cite[Example~2.3]{okada}, 
part (iii) of Proposition~\ref{prop:dirac} would be false in general if the condition that $\mu$ be a Radon measure were relaxed to it being regular; that is, if closed sets were used in place of compact sets $K$.  
%
For a further study of properties of support sets, see e.g.\ \cite{seidel}. 

\begin{proof}[Proof of Proposition~\ref{prop:dirac}] For brevity, let $S_A:=\supp\mu_A$. 

\noindent (i):\quad Suppose that $A$ is a $\mu$-atom, whereas $S_A$ contains two distinct points, say $s_1$ and $s_2$. Let $G_1$ and $G_2$ be open sets in $S$ such that $s_1\in G_1$, $s_2\in G_2$, and $G_1\cap G_2=\emptyset$. 
Then $\mu_A(G_1)>0$ and $\mu(A)-\mu_A(G_1)=\mu_A(S\setminus G_1)\ge\mu_A(G_2)>0$, so that $\mu_A(G_1)\in\big(0,\mu(A)\big)$, which contradicts the condition that $A$ is a $\mu$-atom. 
Thus, part (i) of Proposition~\ref{prop:dirac} is proved. 

\noindent (ii):\quad Suppose that indeed $\card S_A\le1$, $A\ne\emptyset$, and $\mu_A(S_A)=\mu(A)$. 
If at that $\mu(A)=0$ then $\mu_A=\la\de_s$ for $\la=0$ and any $s\in A$. 
Suppose now that $\mu(A)>0$. 
Then $0<\mu(A)=\mu_A(S_A)$, whence $S_A\ne\emptyset$, $\card S_A=1$, $S_A=\{s\}$ for some $s\in S$, $\mu_A(\{s\})=\mu_A(S_A)=\mu(A)>0$, and hence $s\in A$. 
Also, $\mu_A(S\setminus\{s\})=\mu_A(S)-\mu_A(\{s\})=\mu_A(S)-\mu_A(S_A)=0$, whence $\mu_A(B)=0$ for all $B\in\Si$ such that $s\notin B$. 
On the other hand, for all $B\in\Si$ such that $s\in B$ one has $\mu(A)\ge\mu_A(B)\ge\mu_A(\{s\})=\mu_A(S_A)=\mu(A)$, whence $\mu_A(B)=\mu(A)$. 
Thus, for any $B\in\Si$ one has $\mu_A(B)=0$ if $s\notin B$ and $\mu_A(B)=\mu(A)$ if $s\in B$. 
That is, $\mu_A=\la\de_s$ for $\la:=\mu(A)$, which proves part (ii) of Proposition~\ref{prop:dirac}. 


\noindent (iii):\quad 
Note that, if $\mu$ is a Radon measure then $\mu_A$ is so too: if $B\in\Si$ and $K$ is a compact subset of $B$, then $0\le\mu_A(B)-\mu_A(K)=\mu_A(B\setminus K)\le\mu(B\setminus K)=\mu(B)-\mu(K)$. 
So, in the case when $\Si=\B$, part (iii) of Proposition~\ref{prop:dirac} follows from the well known fact that any Radon measure on $\B$ is support-concentrated; see e.g.\ \cite[page~222]{okada} \big(where the terminology ``$\mu$ has a strong support'' is used in place of ``$\mu$ is support-concentrated''\big).
For the readers' convenience, let us present here an easy proof of the latter fact, which works whenever $\Si\supseteq\B$. 
By 
the definition of the support of a measure, for any $s\in S\setminus S_A$ there is a set $G_s$ such that $G_s$ is open, $s\in G_s\subseteq S$, and $\mu_A(G_s)=0$. 
Take now any compact $K\subseteq S\setminus S_A$. Since $\bigcup_{s\in K}G_s\supseteq K$, there is some finite subset $F$ of $K$ such that $\bigcup_{s\in F}G_s\supseteq K$. 
So, $\mu_A(K)\le\sum_{s\in F}\mu_A(G_s)=0$. 
Thus, $\mu_A(K)=0$ for all compact $K\subseteq S\setminus S_A$, and so, since $\mu_A$ is a Radon measure, $\mu_A(S\setminus S_A)=0$ or, equivalently, $\mu_A(S_A)=\mu(A)$. 
This proves part (iii) of Proposition~\ref{prop:dirac}. 

\noindent (iv):\quad For part (iv) of the proposition, see e.g.\ \cite[P16, page XIII]{topsoe}.  

\noindent (v):\quad Part (v) of the proposition is trivial. 
\end{proof}

Proposition~\ref{prop:dirac} immediately yields 

\begin{corollary}\label{cor:dirac,Polish} 
Suppose that  
$S$ is a Polish space and $\Si=\B$. 
Take any $\mu\in\msi$ and $A\in\Si$. 
Then 
$A$ is a $\mu$-atom iff $\mu_A=\la\de_s$ for some $s\in A$ and $\la\in[0,\infty)$. 
\end{corollary} 

Now Theorems~\ref{th:necess} and \ref{th:suff} immediately imply 

\begin{corollary}\label{cor:dirac-ext} 
Suppose that  
$S$ is a Polish space, $\Si=\B$, and $\La=\msi$. 
Then $\ex\La_{\f,c}=\Dekf\cap\La_{\f,c}$ for all $c\in\R^k$. 
\end{corollary} 

This latter result should be enough for most applications. Yet, one may want to compare it with equality \eqref{eq:dirac opt eq} in Corollary~\ref{cor:dirac opt}, which holds without any topological assumptions.

In conclusion, let us briefly discuss 
existing literature. 
The present paper was mainly motivated by the work of Winkler \cite{winkler88}, especially by  
the principal result there:  

\begin{theorem}\emph{\bf (\cite[Theorem~2.1]{winkler88}).} \label{th:winkler}
Suppose that 
the set $\Pi$ of all probability measures in $\La$ is a Choquet-simplex and 
$\ex\Pi\subseteq\De^{(1)}$. 
Then the following conclusions hold. 
\begin{enumerate}[(a)]
	\item 
If $C=
(-\infty,c_1]\times\dots\times(-\infty,c_k]$ for some $(c_1,\dots,c_k)\in\R^k$, then \break 
$\ex(\Pi\cap\La_{\f,C})\subseteq\De_{(1,\f)}^{(1+k)}$, where $(1,\f):=(1,f_1,\dots,f_k)$.   
\item 
For any $c\in
\R^k$, one has $\ex(\Pi\cap\La_{\f,c})=\Pi\cap\La_{\f,c}\cap\De_{(1,\f)}^{(1+k)}$.  
\end{enumerate}
\end{theorem}
By a remark in \cite[Section~9]{phelps}, the meaning of the condition that $\Pi$ is a Choquet-simplex 
can be expressed 
as follows: $\Pi$ is a convex set of probability measures such that for the cone $\C:=[0,\infty)\,\Pi$ generated by $\Pi$ and any $\mu_1$ and $\mu_2$ in $\C$ there is some $\nu\in\C$ such that $\mu_1-\nu$ and $\mu_2-\nu$ are in $\C$ and for any $\tnu\in\C$ such that $\mu_1-\tnu$ and $\mu_2-\tnu$ are in $\C$ one has $\nu-\tnu\in\C$. 


To an extent, our Theorems~\ref{th:necess} and \ref{th:suff} (cf.\ also Proposition~\ref{prop:dirac} and Corollaries~\ref{cor:dirac,Polish} and \ref{cor:dirac-ext}) correspond to  
parts (a) and (b), respectively, of Theorem~\ref{th:winkler}. 
One may note that the second condition, $\ex\Pi\subseteq\De^{(1)}$, in Theorem~\ref{th:winkler} is of the same form (corresponding to $k=0$ affine restrictions on the measure in addition to the requirement that it be a probability measure) as the conclusion $\ex(\Pi\cap\La_{\f,C})\subseteq\De_{(1,\f)}^{(1+k)}$ in part (a) of the theorem, where $k$ additional affine restrictions on the measure are present. That is in distinction with Theorems~\ref{th:necess} and \ref{th:suff} of this paper (cf.\ also Corollaries~\ref{cor:necess,k=1} and \ref{cor:suff,k=1}). 
Moreover, the set $C$ in Theorem~\ref{th:necess} may be any subset of $\R^k$ and not necessarily of the orthant form assumed in part (a) of Theorem~\ref{th:winkler}. 

Also in distinction with Theorem~\ref{th:winkler}, the measures in this paper are not required to be probability ones; for instance, 
the set $\La$ may coincide with the set $\msi$ of all measures on $\Si$. 

Condition~\eqref{eq:cond} in Theorem~\ref{th:necess} appears to be easier to check than the conditions in Theorem~\ref{th:winkler} that $\Pi$ be a Choquet-simplex and $\ex\Pi\subseteq\De^{(1)}$. In particular, \eqref{eq:cond} is trivially satisfied in the just mentioned case when $\La=\msi$, 
which appears to be of main interest in applications; condition
\eqref{eq:cond} holds as well in the examples (a), (b), (c) in \cite[page~585]{winkler88} --  if one replaces there the sets $P$ of probability measures by the corresponding cones $[0,\infty)\,P$. 
At this point one may also recall that condition \eqref{eq:cond} (or, in fact, any other special condition) is not needed or used to establish Theorem~\ref{th:suff}; however, the latter theorem appears not nearly as useful as Theorem~\ref{th:necess} in such applications as Corollary~\ref{cor:dirac opt}.  

One might also want to compare condition \eqref{eq:dirac in} that 
$\La$ contain the set $\Dekf$ (or $\Dek$) of discrete measures with the condition in Theorem~\ref{th:winkler} that the set $\ex\Pi$ be contained in 
$\De^{(1)}$, 
even though these two conditions go in opposite directions.  
It appears that \eqref{eq:dirac in} is generally easier to satisfy and check. 
Note also that condition \eqref{eq:dirac in} is used in this paper only to obtain the equality \eqref{eq:dirac opt eq}. 

One can construct an example when $\Pi$ is a Choquet-simplex 
with $\ex\Pi\subseteq\De^{(1)}$ 
while condition~\eqref{eq:cond} fails to hold for the corresponding cone $\La=[0,\infty)\Pi$. For instance, one may let $\Pi$ be the set of all mixtures of the discrete probability distributions on $\R$ and the absolutely continuous probability distributions on $\R$ with everywhere continuous densities. 

It is also easy to give an example when conditions~\eqref{eq:cond} and \eqref{eq:dirac in} both hold while $\ex\Pi\not\subseteq\De^{(1)}$. For instance, suppose that $S$ is any uncountable set, $\Si$ is the $\si$-algebra over $S$ generated by all countable subsets of $S$, and $\La=\msi$, so that $\Pi$ is the set of all probability measures on $\Si$. Then \eqref{eq:cond} and \eqref{eq:dirac in} are both trivially satisfied. On the other hand, consider the $0,\!1$ measure (say $\pi$) on $\Si$ that takes the value $0$ precisely on all countable sets in $\Si$. Then $\pi\in\ex\Pi$, by Corollary~\ref{cor:suff,k=1}. Yet, $S\setminus\{s\}\in\Si$ and $\pi(S\setminus\{s\})=1$ for any $s\in S$ and hence $\pi\notin\De^{(1)}$. 

However, these two examples may seem rather artificial. It appears that the results given here will be about as effective as those in \cite{winkler88} in most applications. 

The methods presented in 
this paper seem 
different from and more elementary than those of \cite{winkler88}. 
In particular, 
the present paper is self-contained, except for quoting \cite{winkler88} 
and \cite{topsoe} concerning part (v) of Proposition~\ref{prop:=ex} and part 
(iv) of Proposition~\ref{prop:dirac}, respectively.

As pointed out in \cite{winkler88}, the results there generalize ones in Richter \cite{richter57} (for $S\subseteq\R$ and piecewise-continuous $f_j$'s), Mulholland and Rogers \cite{mulhol-rogers} (for $S=\R$), and Karr \cite{karr} (for compact metric spaces $S$). 

An equality similar to \eqref{eq:dirac opt eq} was given by Hoeffding~\cite{hoeff-extr} for $S=\R$; in fact, the result there holds for product measures on $\R^n$. 
When $S$ is an interval in $\R$ and the functions $f_1,\dots,f_k,g$ form a Tchebycheff system, such results can be considerably improved: in that case, the support of extremal measures consists of only about $k/2$, rather than $k$, points; see e.g.\ \cite{karlin-studden,krein-nudelman,tcheb}.




%


\bibliographystyle{abbrv}


\bibliography{C:/Users/Iosif/Dropbox/mtu/bib_files/citations}

\end{document}